\documentclass[11pt]{article}
\usepackage{fullpage} 

\usepackage{authblk}
\usepackage{ucs}
\usepackage{amssymb}
\usepackage{amsthm}
\usepackage{amsmath}
\usepackage{latexsym}
\usepackage[cp1251]{inputenc}
\usepackage[english]{babel}
\usepackage{graphicx}
\usepackage{wrapfig}
\usepackage{caption}
\usepackage{subcaption}
\usepackage{txfonts}
\usepackage{lmodern}
\usepackage{mathrsfs}
\usepackage{algorithm}
\usepackage{dsfont}
\usepackage{enumerate}
\usepackage{hyperref}
\usepackage{multicol}
\usepackage{csquotes}
\usepackage{stmaryrd}
\usepackage{tikz}
\usepackage{comment}
\usepackage{enumitem}

\usepackage{algorithm}
\usepackage{algorithmic}

\usepackage{cite}

\newcommand{\badrounding}[1]{}

\usepackage[initials]{amsrefs}

\newtheorem{theo}{Theorem}

\newtheorem{prob}[theo]{Question}
\newtheorem{lemma}[theo]{Lemma}

\newtheorem{corl}[theo]{Corollary}
\newtheorem{conj}[theo]{Conjecture}

\theoremstyle{definition}

\numberwithin{theo}{section}

\tikzset{
vtx/.style={inner sep=1.1pt, outer sep=0pt, circle, fill,draw}
}

\usepackage{xparse,xpatch,xcolor,tikz}
\usepackage{pgfmath}
\usepackage{xifthen}
\usetikzlibrary{decorations.text, arrows.meta,calc,shadows.blur,shadings}
\usetikzlibrary{shapes.geometric, arrows, positioning}

\title{Rainbow Subgraphs in Edge-colored Complete Graphs - \\ Answering two Questions by Erd\H{o}s and Tuza}

\begin{document}

\author{Maria Axenovich \footnote{\texttt{maria.aksenovich@kit.edu} Research is partially supported by DFG grant FKZ AX 93/2-1.}\qquad Felix Christian Clemen \footnote{\texttt{felix.clemen@kit.edu}}\\Karlsruhe Institute of Technology, 76133 Karlsruhe, Germany}

\maketitle


\abstract{An edge-coloring  of a  complete  graph  with a set of colors $C$ is called {\it completely balanced}   if any vertex is incident to the same number of edges of each color  from $C$. Erd\H{o}s and Tuza asked in $1993$ whether for any graph $F$  on $\ell$ edges and any completely balanced coloring  of any sufficiently  large complete graph using $\ell$ colors contains a rainbow copy of $F$. This question was restated by Erd\H{o}s in his list of  ``Some of my favourite problems on cycles and colourings''. We answer this question in the negative for most cliques $F=K_q$ by giving explicit constructions of respective completely balanced colorings. Further, we answer a related question concerning completely balanced  colorings of complete graphs with more colors than the number of edges in the graph $F$.
}

\section{Introduction}
Let $F$ and $G$ be graphs.
We say that an edge-coloring of $G$ contains a \emph{rainbow} $F$ if $G$ contains a subgraph isomorphic to $F$ such that all edges are assigned distinct colors. 
The existence of  a rainbow $F$  in a ground graph $G$ could be forced by simply using a lot of colors,  by requiring that each vertex of $G$ is incident to sufficiently many colors, or by making sure that each vertex of $G$ is not incident to too many edges of the same color.  These coloring conditions are referred to as anti-Ramsey or locally anti-Ramsey and it is  assumed that the number of colors used on the edges of $G$ is larger that the number of edges in $F$. 
The following list gives just a small sample of references for these and related problems: ~\cites{AJMP,KMSV,AJT,SS,MBNL,RT,APS}. 
Erd\H{o}s and Tuza~\cite{ErdTu} studied the existence of  a rainbow subgraph $F$  in  edge-colored complete graphs when the total number of  colors  is equal to  the number of edges of $F$. Here,  we focus on this problem.\\

Denote by $K_n$ the complete graph on $n$ vertices. An $(\ell,d)$-coloring of $K_n$ is an assignment of colors to edges such that in total $\ell$ colors are used and for every vertex there are at least $d$ edges incident to it, in every color. Let $F$ be a graph with $\ell$ edges. Define $d(n,F)=\infty$ if $K_n$ has an $(\ell,\lfloor (n-1)/\ell \rfloor)$ edge-coloring without a rainbow $F$; otherwise $d(n,F)$ is defined to be the smallest integer $d$ such that every $(\ell,d)$-coloring of $K_n$ contains a rainbow copy of $F$. \\

Erd\H{o}s and Tuza~\cite{ErdTu} determined $d(n,K_3)$ precisely and found an infinite class of graphs $F$  on $\ell$ edges, for which $d(n,F)=\infty$ for every positive $n\equiv 0 \mod \ell$. They~\cite{ErdTu} stated the following question on edge-colorings of the complete graph (Problem 1 in\cite{ErdTu}), also restated by Erd\H{o}s in his  list  of  ``Some of my favourite problems on cycles and colourings'',   \cite{Erdrain}.

\begin{prob}[Erd\H{o}s, Tuza~\cite{ErdTu}]
\label{Erdos1}
Is $d(n,F)$ finite for every graph $F$ on $\ell$ edges and every sufficiently large $n \equiv 1 \mod \ell$?
\end{prob}

If $n-1$ is divisible by $\ell$, we call an $(\ell, (n-1)/\ell)$-coloring of $K_n$ {\it completely  balanced}. Note that  for a graph $F$ on $\ell$ edges and $n-1$ divisible by $\ell$, $d(n,F)=\infty$ if and only if there is a completely balanced  coloring of $K_n$ using $\ell$ colors and containing no rainbow  $F$. We prove that 'most' cliques provide a negative answer to Question~\ref{Erdos1}.\\

Let $S(N)$ be the set of all natural $q$'s such that $4\leq q\leq N$ and for any  $n_0$,  there  is $n\geq n_0$, $n  \equiv 1 \mod \ell$ and a balanced coloring of $K_n$ in $\binom{q}{2}$ colors with no rainbow copy of $K_q$.  Question  \ref{Erdos1}  in case when $F$ is a clique  asks whether $S(N)= \emptyset$ for any natural $N$.  We show that actually not only $S(N)$ is non-empty, but also that it is close to  having size   $N$.

\begin{theo}
\label{mostcounter} $|S(N)| = N - (1+o(1))\frac{N}{\log N}.$
\end{theo}

For the proof of Theorem~\ref{mostcounter} we establish a connection between Question~\ref{Erdos1} for cliques and the Prime Power Conjecture  on perfect difference sets (Conjecture~\ref{PPC}).
We conjecture that in fact  when $F$ is  any  clique of size at least four,  the  answer to Question~\ref{Erdos1} is negative:
\begin{conj}
\label{conjecturecliques}
$S(N)= \{n\in \mathbb{N}: n\geq 4\}$. 
\end{conj}
In further partial support of Conjecture~\ref{conjecturecliques}, we show it for all cliques of size $q\geq 4$ with odd number of edges. 
\begin{theo}
\label{maintheorem}
Let $q\geq 10$ be an integer satisfying $q\equiv 2$ or $3 \mod 4$, and let $\ell=\binom{q}{2}$. For every $k\geq 1$ and $n=(\ell+1)^k$ there exists a
completely balanced edge-coloring of $K_n$ with $\ell$ colors without a rainbow $K_q$, i.e. $d(n,K_q)=\infty$.
\end{theo}
We remark that Theorem~\ref{maintheorem} can be extended to hold for $q=6,7$, however, this requires a more careful analysis of our construction which we omit. 

Erd\H{o}s and Tuza~\cite{ErdTu} also asked the following question in the setting where those edge-colorings of $K_n$ use more colors than the number of edges in $F$.

\begin{prob}[Erd\H{o}s, Tuza~\cite{ErdTu}]
\label{Erdos2}
For a fixed positive integer $\ell$ and any sufficiently large integer  $n$, does every $(\ell+1,\lfloor (n-1)/(\ell+1)\rfloor)$ edge-coloring of $K_n$ contain every graph $F$ on $\ell$ edges as a rainbow subgraph?
\end{prob}
Tuza repeated both questions in \cite{Tuza} and remarked that he expects the answer to Question~\ref{Erdos2} to be affirmative. We answer it in the negative. 

\begin{theo}
\label{corl:maintheorem2}
Let $q\geq 8$ be an integer satisfying $q\equiv 0$ or $1 \mod 4$, and let $\ell=\binom{q}{2}$. For every $k\geq 1$ there exists 
completely balanced edge-coloring of $K_n$ with $\ell+1$ colors, for  $n={(\ell+2)^k}$,  without a rainbow $K_q$.
\end{theo}

Our paper is organized as follows.
In Section~\ref{lexico} we introduce the so-called lexicographical product of colorings which we will use for all our constructions. In Section~\ref{sec:Corls} we prove Theorems~\ref{maintheorem} and \ref{corl:maintheorem2}, and finally in Section~\ref{difference} we prove Theorem~\ref{mostcounter}.


\section{Iterated lexicographical product colorings}
\label{lexico}
For a natural number $n$, let $[n]=\{1, \ldots, n\}$.
For sets of colors $C_1$ and $C_2$, and edge-colorings $c_1:E(K_n)\rightarrow C_1$ and $c_2:E(K_m)\rightarrow C_2$, we define the \emph{lexicographical product coloring} $c_1\times c_2 : E(K_{nm})\rightarrow C_1\cup C_2$ in the following way. Let the vertex set of $K_{nm}$ be the set of pairs $(i,j)$ with $i\in [m]$ and $j\in [n]$ and define
\begin{align*}
    (c_1\times c_2) ((i_1,j_1),(i_2,j_2))=
    \begin{cases}
    c_2(j_1,j_2),& \text{ if } i_1=i_2, j_1\neq j_2, \\
    c_1(i_1,i_2),& \text{ if } i_1\neq i_2,
    \end{cases}
\end{align*}
for $i_1,i_2\in [m]$ and $j_1,j_2\in [n]$ satisfying $(i_1,j_1)\neq (i_2,j_2)$. Lexicographic products have been used in Ramsey theory, see e.g.~\cite{Abbott}. The following lemma shows that taking lexicographic products maintains the properties  of not containing rainbow cliques and being  completely balanced.  

\begin{lemma}
\label{lexicoproduct}
Let $n,m,q\geq 3$ be positive integers and $C$ be a set of colors. Further, let $c_1:E(K_n)\rightarrow C$ and $c_2:E(K_m)\rightarrow C$ be balanced edge-colorings without a rainbow $K_q$. Then $c_1\times c_2$ also is a balanced coloring without a rainbow $K_q$.
\end{lemma}
\begin{proof}
Clearly, $c_1\times c_2$ is a balanced coloring: If in $c_1$ every vertex is incident  $k_1$  edges of every color and in $c_2$ every vertex is incident to  $k_2$ edges of every color, then in $c_1 \times c_2$ every vertex is incident $k_2+k_1m$  edges of every color.\\

Let $S\subseteq V(K_{mn})$ be a set of $q$ vertices.
If all $q$ vertices have the same first coordinate, then $G[S]$ is colored according to the coloring $c_2$, and thus, $S$ is not rainbow. If all $q$ vertices have different values for their first coordinate, then $G[S]$ is colored according to the coloring $c_1$, and thus, $S$ is not rainbow. Otherwise, there are three vertices $x=(x_1,x_2),y=(y_1,y_2),z=(z_1,z_2)\in S$ such that $x_1=y_1\neq z_1$. Then $(c_1\times c_2)(x,z)=c_1(x_1,z_1)=c_1(y_1,z_1)=(c_1\times c_2)(y,z)$ and therefore $S$ is not rainbow. We conclude that the coloring $c_1\times c_2$ does not contain a rainbow $K_q$, completing the proof.
\end{proof}
Iteratively applying Lemma~\ref{lexicoproduct} to the same coloring, we obtain the following. 
\begin{lemma}
\label{iteralexico}
If there exists a completely balanced edge-coloring of $K_n$ with $\ell$ colors and no rainbow $K_q$, then for every $k\geq 1$ there exists a completely balanced edge-coloring of $K_{n^k}$ with $\ell$ colors and no rainbow $K_q$. In particular,
if $d(n, K_q)=\infty$ for integers $n$ and $q$, then $d(n^k, K_q)=\infty$ for all $k\geq 1$.
\end{lemma}
 Lemma~\ref{iteralexico} says that, in order to show that a clique $K_q$ is a negative example to Questions~\ref{Erdos1} or \ref{Erdos2}, it is sufficient to find the desired coloring for a single value of $n$.

\section{The proof of Theorems~\ref{maintheorem} and \ref{corl:maintheorem2}.}
\label{sec:Corls}
First, we consider a construction, that we shall use for  both theorems,  and show some of its properties.


\subsection{The construction}
\label{The coloring}
For a fixed odd integer $\ell$, $\ell\geq 3$, we define an edge-coloring $c$ of $K_{(\ell+1)}$ with vertex set $\{0,1, \ldots, \ell\}$ as follows:
\begin{align}
\label{coloring}
    c(i,j)=\begin{cases}
    i+i \mod \ell & \text{if $j=\ell$}, \\ 
    i+j \mod \ell & \text{otherwise},
    \end{cases}
\end{align}
for $0\leq i<j\leq \ell$. 

We remark that this  coloring was known already over a hundred years ago, see for example~\cite{Lucas}  and is a standard example of a so-called $1$-factorization of the complete graph, i.e. a decomposition of the complete graph into perfect matchings. Informally, the coloring \eqref{coloring} corresponds to  arranging  vertices   from $\{0,1, \ldots, \ell-1\}$ as the corners of a regular $\ell$-gon  in the plane and placing the vertex $\ell$ in the center of the $\ell$-gon. Every color class consists of an edge from the center vertex $\ell$ to a vertex together with all possible perpendicular edges.
See Figure~\ref{coloringK16} for an illustration of this coloring when $\ell=15$. Note that every color class in the coloring \eqref{coloring} is a perfect matching. The coloring can be used as a schedule of competitions with an even number of competitors, in which each contestant plays a game every round and additionally meets every other competitor exactly one time.

\begin{figure}[h!]
    \centering
\begin{tikzpicture}
\path (0,1.2) -- (0,-1.1); 
\draw
\foreach \i in {0,1,...,14}{
(90+24*\i:2.5) coordinate(\i) node[vtx]{}
};
\draw (0:0) coordinate(15) node[vtx]{};

\newcommand{\ovcolnum}[1]
{
    \ifthenelse{#1 = 0}
    {\colorlet{bcolor}{red}}
    {
        \ifthenelse{#1 = 1}
        {\colorlet{bcolor}{lime}}
        {
            \ifthenelse{#1 = 2}
            {\colorlet{bcolor}{orange}}
            {
                \ifthenelse{#1 = 3}
                {\colorlet{bcolor}{yellow}}
                {
                    \ifthenelse{#1 = 4}
                    {\colorlet{bcolor}{brown}}
                    {
                        \ifthenelse{#1 = 5}
                        {\colorlet{bcolor}{green}}
                        {
                            \ifthenelse{#1 = 6}
                            {\colorlet{bcolor}{pink}}
                            {
                            \ifthenelse{#1 = 7}
                            {\colorlet{bcolor}{lightgray}}
                            {
                            \ifthenelse{#1 = 8}
                            {\colorlet{bcolor}{cyan}}
                            {
                            \ifthenelse{#1 = 9}
                            {\colorlet{bcolor}{magenta}}
                            {
                            \ifthenelse{#1 = 10}
                            {\colorlet{bcolor}{olive}}
                            {
                            \ifthenelse{#1 = 11}
                            {\colorlet{bcolor}{violet}}
                            {
                            \ifthenelse{#1 = 12}
                            {\colorlet{bcolor}{darkgray}}
                            {
                                \ifthenelse{#1 = 13}
                                {\colorlet{bcolor}{teal}}
                                {
                                    \ifthenelse{#1 = 14}
                                    {\colorlet{bcolor}{blue}}
                                    {
                                        \colorlet{bcolor}{purple}
                                    }   
                                }   
                            }}}}}}}   
                        }   
                    }   
                }   
            }   
        }
    }
}

\foreach \j in {0,1,...,14}{
\ovcolnum{\j};
\draw[color=bcolor] (\j) to (15);}

\foreach \j in {0,1,...,14}{
\ovcolnum{\j};
\foreach \i in {1,2,3,4,5,6,7}{
\draw[color=bcolor] (90+24*\i+24*\j:2.5)  to (90+24*\j-24*\i:2.5);}}
\end{tikzpicture}
    \caption{The edge-coloring $c: E(K_{16}) \rightarrow [14] \cup\{0\}$ as defined in \eqref{coloring} when $\ell=15$.}
    \label{coloringK16}
\end{figure}
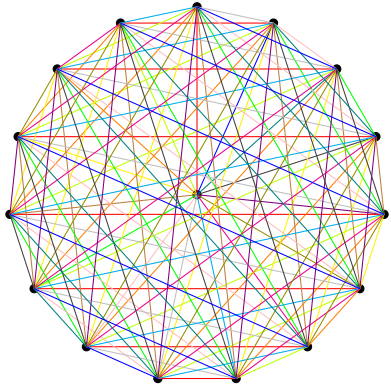

\subsection{Properties of the construction}
\label{verthecol}

We use theory about Sidon sets in abelian groups to prove that there is no rainbow clique of size roughly $\sqrt{\ell}$ in the edge-coloring $c$, defined in \eqref{coloring}. Given an abelian group $G$ and $A\subseteq G$, define  
\begin{align*}
    r_A(x):=|\{(a_1,a_2): \ a_1,a_2\in A, a_1+a_2=x\}|
    \end{align*}
and
\begin{align*}
    r'_A(x)=|\{(a_1,a_2): \ a_1,a_2\in A, a_1\neq a_2, a_1+a_2=x\}|.
\end{align*}
A set $A\subseteq G$ is called 2-\emph{Sidon-set} if $r_A(x)\leq 2$ for all $x\in G$, and it is called weak 2-\emph{Sidon set} if $r'_A(x)\leq 2$ for all $x\in G$. Cilleruelo, Ruzsa and Vinuesa [Corollary 2.3. in \cite{Cilleruelo}] proved that a weak 2-Sidon set $A\subset \mathbb{Z}_\ell$, where $\ell$ is odd, satisfies 
\begin{align}
\label{weakSidon}
    |A|\leq \sqrt{\ell}+\frac{5}{2}.
\end{align}
Bajnok [Proposition C.7 in \cite{Bajnok}] proved that for a 2-Sidon set $A\subset \mathbb{Z}_\ell$, 
\begin{align}
\label{Sidon}
    |A|\leq \frac{\sqrt{4\ell-3}+1}{2}.
\end{align}

The following lemma establishes a connection between rainbow cliques in the coloring $c$ and Sidon sets in $\mathbb{Z}_\ell$. 
A set $S\subseteq V(G)$ is  called \emph{rainbow} if all edges in $G[S]$ is  rainbow. 
\begin{lemma}
\label{Sidonrainbow}
Let $\ell$ be an odd integer, $\ell \geq 3$,  and $S\subseteq V(K_{\ell+1})$ be rainbow in the coloring $c: E(K_{\ell+1}) \rightarrow \{0, 1, \ldots, \ell-1\}$
as defined in \eqref{coloring}. If $\ell\in S$, then $S\setminus \{\ell\}$ is a $2$-Sidon set in $\mathbb{Z}_\ell$, otherwise $S$ is a weak $2$-Sidon set in $\mathbb{Z}_\ell$.  
\end{lemma}
\begin{proof}
In this proof addition will be in $\mathbb{Z}_\ell$.

First, let $\ell\in S$ and define $S'=S\setminus \{\ell\}$. Assume, towards contradiction, that there exists $x\in \mathbb{Z}_{\ell}$ such that $r_{S'}(x)\geq 3$, i.e. 
$x=a_1+b_1=a_2+b_2=a_3+b_3$, for three distinct pairs $(a_i,b_i)$, $a_i,  b_i\in S'$ and $i=1,2, 3$. 
Assume first that $a_i=b_i$, for some $i$, say for $i=1$.  Since $\ell$ is odd, $a_1+a_1\neq a_i+a_i$ for $a_i\neq a_1$,  so 
we have without loss of generality that $b_2\neq a_1$. Then 
since $a_1+a_1=a_2+b_2$, we have  $c(a_2,b_2)=a_2+b_2=a_1+a_1=c(a_1,\ell)$, 
contradicting that $S$ is rainbow. We conclude that for each $i=1,2,3$, $a_i\neq b_i$. Since $(a_i, b_i)$ are distinct pairs $i=1,2,3$, without loss of generality $\{a_1, b_1\}\neq \{a_2, b_2\}$.
By the definition of the coloring \eqref{coloring}, $c(a_1,b_1)=c(a_2,b_2)$, contradicting that $S$ is rainbow. We conclude that $S'=S\setminus \{\ell\}$ is a $2$-Sidon set.

Now, let $\ell\notin S$. Assume, towards a contradiction, that there exists $x\in \mathbb{Z}_{\ell}$ such that $r'_S(x)\geq 3$, i.e. for three distinct  pairs $(a_i,b_i)$, $i=1, 2,  3$  satisfying $a_i,b_i\in S, a_i\neq b_i$, we have $a_i+b_i=x$. By the same argument as before, this contradicts that $S$ is rainbow. We conclude that $S$ is a weak $2$-Sidon set.
\end{proof}

\begin{lemma}
\label{K16israinbowk6-free}
Let $\ell$ be an odd integer.
The coloring $c: E(K_{\ell+1}) \rightarrow \{0, 1, \ldots, \ell-1\}
$ as defined in \eqref{coloring} is a completely balanced coloring that does not contain a rainbow $K_m$, where $m=\lfloor \sqrt{\ell}+\frac{7}{2}\rfloor$.
\end{lemma}
\begin{proof}
Since every vertex is incident to exactly one edge in every color, the coloring $c$ is completely balanced.

Assume that there exists a rainbow $K_m$ on some vertex set $T\subseteq V(K_{\ell+1})$ in the edge-coloring $c$ of $E(K_{\ell+1})$. If $\ell\in S$, then $S\setminus\{\ell\}$ is a 2-Sidon set by Lemma~\ref{Sidonrainbow}. Therefore, by \eqref{Sidon}, we get
\begin{align*}
    \left\lfloor \sqrt{\ell}+\frac{5}{2}\right\rfloor=m-1=|S\setminus\{\ell\}|\leq \left\lfloor\frac{\sqrt{4\ell-3}+1}{2}\right\rfloor\leq \left\lfloor\sqrt{\ell}+\frac{1}{2}\right\rfloor.
\end{align*}
Thus $m<\lfloor \sqrt{\ell}+\frac{7}{2}\rfloor$. We can assume that $\ell\notin S$. The set $S \subset \mathbb{Z}_{\ell}$ is a weak $2$-Sidon set by Lemma~\ref{Sidonrainbow}. Therefore, by \eqref{weakSidon}, we get
\begin{align*}
    \left\lfloor \sqrt{\ell}+\frac{7}{2}\right\rfloor=m=|S|\leq \left\lfloor \sqrt{\ell}+\frac{5}{2}\right\rfloor. 
\end{align*}
Thus $m<\lfloor \sqrt{\ell}+\frac{7}{2}\rfloor$.  We conclude that there is no rainbow $K_m$ in the edge-coloring $c$ of $K_{\ell+1}$, for $m=\lfloor \sqrt{\ell}+\frac{7}{2}\rfloor$.
\end{proof}

\subsection{Deducing Theorems~\ref{maintheorem} and \ref{corl:maintheorem2}  }
We prove the following Theorem which implies both Theorems~\ref{maintheorem} and \ref{corl:maintheorem2} quickly, and in fact provides many examples of graphs, asides form cliques, answering Question~\ref{Erdos1} and \ref{Erdos2} in the negative.
\begin{theo}
\label{theo:maintheorem}
Let $\ell\geq 3$ be an odd integer. 
For every integer $k\geq 1$ and $n= (\ell+1)^k$ there is a completely balanced coloring of $K_n$ with $\ell$ colors without a rainbow $K_m$, where $m=\lfloor \sqrt{\ell}+\frac{7}{2}\rfloor$.
\end{theo}

\begin{proof}[Proof of Theorem~\ref{theo:maintheorem}]
By Lemma~\ref{iteralexico} it is sufficient to find such a coloring for $k=1$. By Lemma~\ref{K16israinbowk6-free} the coloring defined in \eqref{coloring} has the desired properties.
\end{proof}
See Figure~\ref{coloringK162} for an illustration of the coloring used for proving Theorem~\ref{theo:maintheorem} when $k=2$ and $\ell=15$.

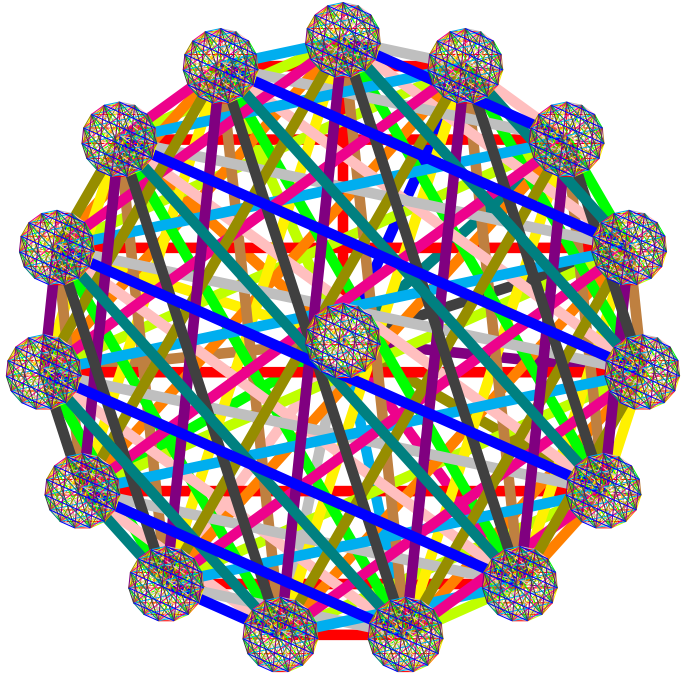
\begin{figure}[ht]
    \centering
\begin{tikzpicture}
\path (0,1.2) -- (0,-1.1); 

\newcommand{\ovcolnum}[1]
{
    \ifthenelse{#1 = 0}
    {\colorlet{bcolor}{red}}
    {
        \ifthenelse{#1 = 1}
        {\colorlet{bcolor}{lime}}
        {
            \ifthenelse{#1 = 2}
            {\colorlet{bcolor}{orange}}
            {
                \ifthenelse{#1 = 3}
                {\colorlet{bcolor}{yellow}}
                {
                    \ifthenelse{#1 = 4}
                    {\colorlet{bcolor}{brown}}
                    {
                        \ifthenelse{#1 = 5}
                        {\colorlet{bcolor}{green}}
                        {
                            \ifthenelse{#1 = 6}
                            {\colorlet{bcolor}{pink}}
                            {
                            \ifthenelse{#1 = 7}
                            {\colorlet{bcolor}{lightgray}}
                            {
                            \ifthenelse{#1 = 8}
                            {\colorlet{bcolor}{cyan}}
                            {
                            \ifthenelse{#1 = 9}
                            {\colorlet{bcolor}{magenta}}
                            {
                            \ifthenelse{#1 = 10}
                            {\colorlet{bcolor}{olive}}
                            {
                            \ifthenelse{#1 = 11}
                            {\colorlet{bcolor}{violet}}
                            {
                            \ifthenelse{#1 = 12}
                            {\colorlet{bcolor}{darkgray}}
                            {
                                \ifthenelse{#1 = 13}
                                {\colorlet{bcolor}{teal}}
                                {
                                    \ifthenelse{#1 = 14}
                                    {\colorlet{bcolor}{blue}}
                                    {
                                        \colorlet{bcolor}{purple}
                                    }   
                                }   
                            }}}}}}}   
                        }   
                    }   
                }   
            }   
        }
    }
}

\def\k{
\foreach \j in {0,1,...,14}{
\ovcolnum{\j};
\draw[color=bcolor] (90+24*\j:2.5) to (0:0);}

\foreach \j in {0,1,...,14}{
\ovcolnum{\j};
\foreach \i in {1,2,3,4,5,6,7}{
\draw[color=bcolor] (90+24*\i+24*\j:2.5)  to (90+24*\j-24*\i:2.5);}}}

\def\b{4}

\foreach \j in {0,1,...,14}{
\ovcolnum{\j};
\draw[color=bcolor, line width=4 pt] (90+24*\j:3.5) to (90+24*\j:0.5);}

\foreach \j in {0,1,...,14}{
\ovcolnum{\j};
\foreach \i in {1,2,3,4,5,6,7}{
\draw[color=bcolor, line width=4 pt] (90+24*\i+24*\j:4)  to (90+24*\j-24*\i:4);}}

\begin{scope}[xshift=0*\b cm, yshift=1*\b cm, scale=0.2] \k \end{scope}
\begin{scope}[xshift=-0.4067*\b cm, yshift=0.9135*\b cm, scale=0.2] \k \end{scope}
\begin{scope}[xshift=-0.7431*\b cm, yshift=0.6691*\b cm, scale=0.2] \k \end{scope}
\begin{scope}[xshift=-0.9510*\b cm, yshift=0.3090*\b cm, scale=0.2] \k \end{scope}

\begin{scope}[xshift=-0.9945*\b cm, yshift=-0.1045*\b cm, scale=0.2] \k \end{scope}
\begin{scope}[xshift=-0.8660*\b cm, yshift=-0.5*\b cm, scale=0.2] \k \end{scope}
\begin{scope}[xshift=-0.5877*\b cm, yshift=-0.8090*\b cm, scale=0.2] \k \end{scope}
\begin{scope}[xshift=-0.2079*\b cm, yshift=-0.9781*\b cm, scale=0.2] \k \end{scope}

\begin{scope}[xshift=0.2079*\b cm, yshift=-0.9781*\b cm, scale=0.2] \k \end{scope}
\begin{scope}[xshift=0.5877*\b cm, yshift=-0.8090*\b cm, scale=0.2] \k \end{scope}
\begin{scope}[xshift=0.8660*\b cm, yshift=-0.5*\b cm, scale=0.2] \k \end{scope}
\begin{scope}[xshift=0.9945*\b cm, yshift=-0.1045*\b cm, scale=0.2] \k \end{scope}
\begin{scope}[xshift=0.9510*\b cm, yshift=0.3090*\b cm, scale=0.2] \k \end{scope}
\begin{scope}[xshift=0.7431*\b cm, yshift=0.6691*\b cm, scale=0.2] \k \end{scope}
\begin{scope}[xshift=0.4067*\b cm, yshift=0.9135*\b cm, scale=0.2] \k \end{scope}

\begin{scope}[xshift=0 cm, yshift=0 cm, scale=0.2] \k \end{scope}

\end{tikzpicture}
    \caption{The edge-coloring of $K_{16^2}$.}
    \label{coloringK162}
\end{figure}

\begin{proof}[Proof of Theorem~\ref{maintheorem}]
Theorem~\ref{maintheorem} simply follows from Theorem~\ref{theo:maintheorem} by observing that $\ell=\binom{q}{2}$ is odd for $q\equiv 2$ or $3 \mod 4$, and $q\geq m=\left\lfloor \sqrt{\binom{q}{2}}+\frac{7}{2}\right\rfloor$ for $q\geq 10$. By Theorem~\ref{theo:maintheorem} there exists a completely balanced coloring of $K_n$ with $\ell$ colors without a rainbow $K_m$. Since $q\geq m$ this coloring does not contain a rainbow $K_q$.
\end{proof}

\begin{proof}[Proof of Theorem~\ref{corl:maintheorem2}]
Theorem~\ref{corl:maintheorem2} simply follows from Theorem~\ref{theo:maintheorem} by observing that $\ell+1 $
is odd for $q\equiv 0$ or $1 \mod 4$ and that $q\geq \left\lfloor \sqrt{\binom{q}{2}+1}+\frac{7}{2}\right\rfloor$ for $q\geq 8$.
\end{proof}


\section{Proof of Theorem~\ref{mostcounter}}
\label{difference}
To prove Theorem~\ref{mostcounter} we establish a connection between rainbow subsets in a certain coloring and perfect difference sets. 

A subset $A\subseteq \mathbb{Z}_{n}$ is a \emph{perfect difference set} if every non-zero element $a \in \mathbb{Z}_{n}\setminus\{0\}$ can be written uniquely as the difference of two elements from $A$. For example, $\{2,3,5\}$ is a perfect difference set in $\mathbb{Z}_{7}$. If $A$ is a perfect difference set of size $q$, then $n=q^2-q+1$.
The following lemma establishes a connection between perfect difference sets and the quantity $d(K_q,n)$.

\begin{lemma}
\label{2cons}
Let $q\geq 2$. If there is no perfect difference set of size $q$ in $\mathbb{Z}_{q^2-q+1}$, then $d(K_q,n)=\infty$ for infinitely  many values of $n$ of the form $n\equiv 1 \mod \binom{q}{2}$.
\end{lemma}

\begin{proof}
Let $q$ be an integer such that there is no perfect difference set of size $q$ in $\mathbb{Z}_n$, where $n=2\binom{q}{2}+1=q^2-q+1$. Label the vertices of $K_n$ with the elements from $\mathbb{Z}_n$. Now, we color the edges of $K_n$ with colors from $Z_n\setminus\{0\}$ and identify the colors $a$ and $-a$ with each other. An edge $ab$ is simply colored by $a-b$ (which is the same color as $b-a$). This coloring is an $(\binom{q}{2},2)$ edge-coloring of $K_n$. See Figure~\ref{coloringK13} for an illustration of this coloring when $q=4$ and $n=13$. Assume that $A\subseteq \mathbb{Z}_n$ is the vertex set of a rainbow $K_q$ in this coloring. Then $A\subseteq Z_n$ is a perfect difference set of size $q$, a contradiction. Thus, there is no rainbow copy of $K_q$. We conclude $d(n, K_q)=\infty$, and therefore, applying Lemma~\ref{iteralexico} completes the proof.
\end{proof}

We remark that the coloring used for Lemma~\ref{2cons}, which also is displayed in Figure~\ref{coloringK13}, is the standard example of a $2$-factorization of a complete graph with the number of vertices being odd, i.e. an edge-coloring of the complete graph such that every color class is a spanning $2$-regular subgraph. 

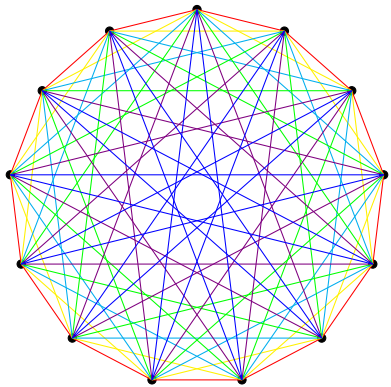
\begin{figure}[h!]
    \centering
\begin{tikzpicture}
\path (0,1.2) -- (0,-1.1); 
\draw
\foreach \i in {1,...,13}{
(90+27.692*\i:2.5) coordinate(\i) node[vtx]{}
};

\newcommand{\ovcolnum}[1]
{
    \ifthenelse{#1 = 0}
    {\colorlet{bcolor}{red}}
    {
        \ifthenelse{#1 = 1}
        {\colorlet{bcolor}{lime}}
        {
            \ifthenelse{#1 = 2}
            {\colorlet{bcolor}{orange}}
            {
                \ifthenelse{#1 = 3}
                {\colorlet{bcolor}{yellow}}
                {
                    \ifthenelse{#1 = 4}
                    {\colorlet{bcolor}{brown}}
                    {
                        \ifthenelse{#1 = 5}
                        {\colorlet{bcolor}{green}}
                        {
                            \ifthenelse{#1 = 6}
                            {\colorlet{bcolor}{pink}}
                            {
                            \ifthenelse{#1 = 7}
                            {\colorlet{bcolor}{lightgray}}
                            {
                            \ifthenelse{#1 = 8}
                            {\colorlet{bcolor}{cyan}}
                            {
                            \ifthenelse{#1 = 9}
                            {\colorlet{bcolor}{magenta}}
                            {
                            \ifthenelse{#1 = 10}
                            {\colorlet{bcolor}{olive}}
                            {
                            \ifthenelse{#1 = 11}
                            {\colorlet{bcolor}{violet}}
                            {
                            \ifthenelse{#1 = 12}
                            {\colorlet{bcolor}{darkgray}}
                            {
                                \ifthenelse{#1 = 13}
                                {\colorlet{bcolor}{teal}}
                                {
                                    \ifthenelse{#1 = 14}
                                    {\colorlet{bcolor}{blue}}
                                    {
                                        \colorlet{bcolor}{purple}
                                    }   
                                }   
                            }}}}}}}   
                        }   
                    }   
                }   
            }   
        }
    }
}

\ovcolnum{0};
\draw[color=bcolor] (1) to (2);
\draw[color=bcolor] (2) to (3);
\draw[color=bcolor] (3) to (4);
\draw[color=bcolor] (4) to (5);
\draw[color=bcolor] (5) to (6);
\draw[color=bcolor] (6) to (7);
\draw[color=bcolor] (7) to (8);
\draw[color=bcolor] (8) to (9);
\draw[color=bcolor] (9) to (10);
\draw[color=bcolor] (10) to (11);
\draw[color=bcolor] (11) to (12);
\draw[color=bcolor] (12) to (13);
\draw[color=bcolor] (13) to (1);

\ovcolnum{3};
\draw[color=bcolor] (1) to (3);
\draw[color=bcolor] (2) to (4);
\draw[color=bcolor] (3) to (5);
\draw[color=bcolor] (4) to (6);
\draw[color=bcolor] (5) to (7);
\draw[color=bcolor] (6) to (8);
\draw[color=bcolor] (7) to (9);
\draw[color=bcolor] (8) to (10);
\draw[color=bcolor] (9) to (11);
\draw[color=bcolor] (10) to (12);
\draw[color=bcolor] (11) to (13);
\draw[color=bcolor] (12) to (1);
\draw[color=bcolor] (13) to (2);

\ovcolnum{8};
\draw[color=bcolor] (1) to (4);
\draw[color=bcolor] (2) to (5);
\draw[color=bcolor] (3) to (6);
\draw[color=bcolor] (4) to (7);
\draw[color=bcolor] (5) to (8);
\draw[color=bcolor] (6) to (9);
\draw[color=bcolor] (7) to (10);
\draw[color=bcolor] (8) to (11);
\draw[color=bcolor] (9) to (12);
\draw[color=bcolor] (10) to (13);
\draw[color=bcolor] (11) to (1);
\draw[color=bcolor] (12) to (2);
\draw[color=bcolor] (13) to (3);

\ovcolnum{5};
\draw[color=bcolor] (1) to (5);
\draw[color=bcolor] (2) to (6);
\draw[color=bcolor] (3) to (7);
\draw[color=bcolor] (4) to (8);
\draw[color=bcolor] (5) to (9);
\draw[color=bcolor] (6) to (10);
\draw[color=bcolor] (7) to (11);
\draw[color=bcolor] (8) to (12);
\draw[color=bcolor] (9) to (13);
\draw[color=bcolor] (10) to (1);
\draw[color=bcolor] (11) to (2);
\draw[color=bcolor] (12) to (3);
\draw[color=bcolor] (13) to (4);

\ovcolnum{11};
\draw[color=bcolor] (1) to (6);
\draw[color=bcolor] (2) to (7);
\draw[color=bcolor] (3) to (8);
\draw[color=bcolor] (4) to (9);
\draw[color=bcolor] (5) to (10);
\draw[color=bcolor] (6) to (11);
\draw[color=bcolor] (7) to (12);
\draw[color=bcolor] (8) to (13);
\draw[color=bcolor] (9) to (1);
\draw[color=bcolor] (10) to (2);
\draw[color=bcolor] (11) to (3);
\draw[color=bcolor] (12) to (4);
\draw[color=bcolor] (13) to (5);

\ovcolnum{14};
\draw[color=bcolor] (1) to (7);
\draw[color=bcolor] (2) to (8);
\draw[color=bcolor] (3) to (9);
\draw[color=bcolor] (4) to (10);
\draw[color=bcolor] (5) to (11);
\draw[color=bcolor] (6) to (12);
\draw[color=bcolor] (7) to (13);
\draw[color=bcolor] (8) to (1);
\draw[color=bcolor] (9) to (2);
\draw[color=bcolor] (10) to (3);
\draw[color=bcolor] (11) to (4);
\draw[color=bcolor] (12) to (5);
\draw[color=bcolor] (13) to (6);


\end{tikzpicture}
    \caption{The edge-coloring of $K_{13}$ with $6=\binom{4}{2}$ colors as defined in Lemma~\ref{2cons}. We remark that it contains a rainbow $K_4$. This figure only serves the purpose of illustrating the coloring.}
    \label{coloringK13}
\end{figure}

Singer~\cite{Sing} constructed perfect difference sets of sizes $p^k+1$, where $p$ is prime and $k\geq1$. The non-existence of perfect difference sets for sizes not of this form is an old question in number theory which has attracted many researchers~\cite{Jung,Sing,evans,BauGor,Gordon,Peluse,Will}. 
\begin{conj}[Prime Power Conjecture]
\label{PPC}
A perfect difference set of size $q$ exists if and only if $q-1$ is a prime power. 
\end{conj}
The Prime Power conjecture was computationally verified for $q\leq 2 \cdot 10^9$ by Baumert and Gordon~\cite{BauGor,Gordon}. 
Various conditions for non-existence of perfect difference sets have been proven. For example, Corollary 1 in \cite{Gordon} provides divisibility conditions leading to the following result.
\begin{corl}
Let $q$ be an integer such that $q-1$ is not divisible by $6, 10, 14, 15, 21, 22, 26, 33, 34$, $35, 38, 39, 46, 51, 55, 57, 58, 62$ or $65$. Then, $d(K_q,n)=\infty$ for infinitely many values of $n$ of the form $n\equiv 1 \mod \binom{q}{2}$.
\end{corl}
\begin{proof}[Proof of Theorem~\ref{mostcounter}]
Recently, Peluse~\cite{Peluse} proved that the number of positive integers $q \leq N$ such that $\mathbb{Z}_{q^2-q+1}$
contains a perfect difference set of size $q$ is $(1+o(1))N/\log N$, which is the same order as the number of prime powers of size at most $N$. Pelusi's result together with Lemma~\ref{2cons} completes the proof of Theorem~\ref{mostcounter}.
\end{proof}

\section*{Acknowledgements} The second author thanks Cameron Gates Rudd for helpful discussions.

\end{document}